\documentclass[a4paper, 11pt, reqno]{amsart}
\usepackage{amssymb,amsthm,amsmath,tikz,setspace,enumerate,enumitem,bbm}
\usepackage[textsize=footnotesize,color=blue]{todonotes}
\usepackage[margin=2.5cm]{geometry}

\newtheorem{theorem}{Theorem}

\newtheorem{lemma}[theorem]{Lemma}
\newtheorem{corollary}[theorem]{Corollary}

\newcommand{\sm}{\setminus}
\newcommand{\N}{\mathbb{N}}
\newcommand{\R}{\mathbb{R}}

\newcommand{\Pro}{\mathbb{P}}
\newcommand{\Exp}{\mathbb{E}}
\newcommand{\mix}{\text{\rm mix}}
\newcommand{\diam}{\text{\rm diam}}

\newcommand{\cA}{\mathcal{A}}
\newcommand{\cB}{\mathcal{B}}

\newcommand{\cC}{\mathcal{C}}
\newcommand{\cG}{\mathcal{G}}
\newcommand{\cF}{\mathcal{F}}

\newcommand{\cS}{\mathcal{S}}
\newcommand{\cI}{\mathcal{I}}
\newcommand{\cH}{\mathcal{H}}

\newcommand{\fS}{\mathfrak{S}}

\newcommand{\bE}{\mathbb{E}}
\newcommand{\bR}{\mathbb{R}}

\newcommand{\cri}{{crit}}

\newcommand{\new}[1]{#1}

\title{Critical percolation on random regular graphs}
\author{Felix Joos}
\author{Guillem Perarnau}
\thanks{The first author was supported by the EPSRC, grant no. EP/M009408/1.}

\begin{document}
\begin{abstract}
We show that for all $d\in \{3,\ldots,n-1\}$ the size of the largest component of a random $d$-regular graph on $n$ vertices around the percolation threshold $p=1/(d-1)$ is $\Theta(n^{2/3})$, with high probability.
This extends known results for fixed $d\geq 3$ and for $d=n-1$, confirming a prediction of Nachmias and Peres on a question of Benjamini. 
As a corollary, for the largest component of the percolated random $d$-regular graph, we also determine the diameter and the mixing time of the lazy random walk.
In contrast to previous approaches, our proof is based on a simple application of the switching method.
%
\end{abstract}
\maketitle


\section{Introduction}

For every $d\in \{3,\ldots,n-1\}$, let $\cG_{n,d}$ be the set of all simple and vertex-labelled $d$-regular graphs on $n$ vertices and let $G_{n,d}$ be a graph chosen uniformly at random from $\cG_{n,d}$. For $p\in[0,1]$, let $G_{n,d,p}$ be a graph obtained from $G_{n,d}$  by retaining each edge independently with probability $p$.
The goal of this paper is to study the order of the largest component of $G_{n,d,p}$, denoted by $L_1(G_{n,d,p})$, in terms of $n,d$ and $p$.

Most of the literature in the area focuses either on fixed $d\geq 3$ or on $d=n-1$. 
Goerdt~\cite{goerdt2001giant} showed the existence of a critical probability, $p_\cri:=1/(d-1)$, such that for every fixed $d\geq 3$ and every $\epsilon>0$ the following holds with probability $1-o(1)$: 
if $p\leq (1-\epsilon)p_\cri$, then $L_1(G_{n,d,p})=O(\log{n})$, while if $p\geq (1+\epsilon)p_\cri$, then $L_1(G_{n,d,p})=\Theta(n)$. 
Similar results were also obtained in a more general setting by Alon, Benjamini and Stacey~\cite{alon2004percolation}. 
For $d=n-1$, the random graph $G_{n,d,p}$ corresponds to the classic Erd\H os-R\'enyi random graph $G_{n,p}$.
In their seminal paper~\cite{erd6s1960evolution}, Erd\H os and R\'enyi proved that for every $\epsilon>0$, the following holds with probability $1-o(1)$: 
if $p\leq (1-\epsilon)/n$, then the largest component of $G_{n,p}$ has order $O(\log n)$, 
if $p= 1/n$ (critical probability), then it has order $\Theta(n^{2/3})$, while if $p\geq (1+ \epsilon)/n$, then it has linear order. 

Both for fixed $d\geq 3$ and for $d=n-1$, the behaviour around the critical probability has attracted a lot of interest. 
It is well established that the critical window in $G_{n,p}$ around $p=1/n$ is of order $n^{-1/3}$~(see e.g.~\cite{nachmias2010critical2}). More precise estimates can be found in~\cite{LPW94}.
Benjamini posed the problem of determining the width of the critical window in $G_{n,d,p}$ around $p_\cri=1/(d-1)$~(see~\cite{nachmias2010critical, pittel2008edge}).
Nachmias and Peres~\cite{nachmias2010critical} and Pittel~\cite{pittel2008edge}, independently showed that the critical window exhibits mean-field behaviour for fixed $d\geq 3$, namely, the following holds with probability $1-o(1)$: 
for every fixed $\lambda\in\bR$, if $p=\frac{1+\lambda n^{-1/3}}{d-1}$, then $L_1(G_{n,d,p})=\Theta(n^{2/3})$. See also Riordan~\cite{riordan2012phase} for more precise results on $L_1(G_{n,d,p})$ in the critical window.

The case when $d$ is an arbitrary function of $n$
is much less understood.
It follows from existing results in the literature\footnote{The non-existence of a linear order component when $p\leq (1-\epsilon)p_\cri$ follows from Proposition~1 in~\cite{nachmias2010critical}. 
The existence of a linear order component when $p\geq (1+\epsilon)p_\cri$ follows from the expansion properties of $G_{n,d}$ (see Corollary 2.8 in~\cite{krivelevich2001random}) and the results on $(n,d,\lambda)-$graphs in~\cite{krivelevich2013phase}.} that for every $d\in \{3,\ldots,n-1\}$, the critical probability for the existence of a linear order component in $G_{n,d,p}$ is $1/(d-1)$.   
Results inside the critical window for given $d$-regular graphs have also been obtained 
in the context of transitive graphs under the finite triangle condition~\cite{borgs2005random} or under certain expansion conditions~\cite{nachmias2009mean}.

Finally, similar results have been obtained for irregular degree sequences whenever the average degree is bounded by a constant~\cite{bollobas2015old,fountoulakis2007percolation, fountoulakis2016percolation, janson2008percolation}. 

\smallskip

In view that both the sparse regime (fixed $d\geq 3$) and the densest one ($d=n-1$) exhibit similar properties, Nachmias and Peres~\cite{nachmias2010critical} suggested that the mean-field behaviour extends to every $d\in \{3,\ldots,n-1\}$. 
In this paper we confirm this prediction in the critical window and thus answer the question posed by Benjamini for all $d\in \{3,\ldots,n-1\}$.


\begin{theorem}\label{thm:main}
Suppose $\lambda\in \mathbb{R}$ and $d,n\in \N$ such that $3\leq d \leq n-1$ and $n$ is sufficiently large.
Let $p=\frac{1+\lambda n^{-1/3}}{d-1}$. 
Then for every sufficiently large $A=A(\lambda)$, we have
$$
\Pro[L_1(G_{n,d,p}) \notin [A^{-1}n^{2/3}, An^{2/3}]]\leq 20A^{-1/2}\;.
$$
\end{theorem}

The upper bound in Theorem~\ref{thm:main} directly follows from the upper bound for $d$-regular graphs in Proposition~1 in~\cite{nachmias2010critical}. 
The proof of the lower bound is more intricate and we devote the rest of the paper to it.

Most of the previous work on the component structure of $G_{n,d,p}$ uses the configuration model introduced by Bollob\'as in~\cite{bollobas1980probabilistic}. The configuration model, denoted by $G_{n,d}^*$, is a model of random $d$-regular multigraphs on $n$ vertices. Conditional on $G_{n,d}^*$ being simple, one obtains the uniform distribution on $\cG_{n,d}$.
It is well-known~(see for example~\cite{wormald1999models}) that
\begin{align}\label{eq:simple}
\Pro[G_{n,d}^* \text{ simple}] = e^{-\Omega(\new{d^2})}\;.
\end{align}
While $\Pro[G_{n,d}^* \text{ simple}]$ is constant for fixed $d\geq 3$,
it quickly tends to $0$ if $d$ grows with $n$, and new ideas are needed to study  $G_{n,d}$.
A standard tool to estimate probabilities for $G_{n,d}$ when $d$ grows with $n$ is the switching method, introduced by McKay in~\cite{mckay1985asymptotics}.
For instance, this method has been used to estimate~\eqref{eq:simple} for $d=o(\sqrt{n})$~\cite{mckay1991asymptotic} or to determine several combinatorial properties of $G_{n,d}$ when $d$ grows with $n$~\cite{krivelevich2001random}.

The proof of the lower bound in Theorem~\ref{thm:main} is based on the
analysis of an exploration process in $G_{n,d,p}$ using the switching method. \new{The central quantity that we track through the process is the number of edges between the explored and unexplored parts of the graph, denoted by $X_t$. Our proof relies on sharp estimations of the first and second moments of $X_t$.}

This approach is inspired by recent developments of the switching method for the study of the component structure of random graphs with a given degree sequence~\cite{fountoulakis2016percolation,joos2016how}.
We take this opportunity to illustrate the use of our method with a simple proof that makes no assumptions on $d$.


\medskip

\new{\textbf{The critical window.} Theorem~\ref{thm:main} shows that the critical window has width $\Omega(n^{-1/3})$. 
Proposition~1 in~\cite{nachmias2010critical} implies that, as $\lambda\to -\infty$, the typical order of the largest component is $o(n^{2/3})$. 
Following analogous ideas as the ones used in the proof of Theorem~\ref{thm:main}, one obtains that, as $\lambda\to \infty$, the typical order of the largest component is $\omega(n^{2/3})$. More precisely, there exist constants $c,C>0$ such that for every $3\leq d\leq n-1$ and $\lambda >0$, if $p=\frac{1+\lambda n^{-1/3}}{d-1}$, then
$$
\Pro\left[L_1(G_{n,d,p}) \leq c\cdot \lambda n^{2/3}\right]\leq C\lambda^{-1}\;.
$$
The proof of this statement is simpler than the proof of our main theorem, since the assumption $\lambda>0$ implies that $X_t$ has positive drift. 
In particular, the first part of the exploration process can be analysed using a first moment argument only 
and for the entire process it suffices to control the variance of $X_t$ from above. It follows that the width of the critical window is $\Theta(n^{-1/3})$.
}

\new{In its current form, our method does not give sharp estimates for $L_1(G_{n,d,p})$ in the barely subcritical and barely supercritical regimes. However, we believe that similar estimates as the ones in Lemma~\ref{lem:exp} hold in general and may be used to extend the results of Nachmias and Peres in~\cite{nachmias2010critical} to all $d\in \{3,\ldots,n-1\}$.} 

\medskip

\new{\textbf{Diameter and Mixing Time.} We present a consequence} of Theorem~\ref{thm:main}.
For a component $\cC$, 
let $\diam(\cC)$ denote its diameter and let $T_\mix(\cC)$ denote the mixing time of the lazy random walk on $\cC$. 
Theorem~1.2 in~\cite{nachmias2008critical} implies the following corollary.



\begin{corollary}\label{cor:diam}
Suppose $\lambda\in \R$ and $d,n\in \N$ such that $3\leq d \leq n-1$ and $n$ is sufficiently large.
Let $p=\frac{1+\lambda n^{-1/3}}{d-1}$. 
Let $\mathcal{C}$ be the largest component of $G_{n,d,p}$.
Then, for every $\epsilon>0$, there exists $A= A(\lambda,\epsilon)$ such that
$$
\Pro[\diam(\mathcal{C}) \notin [A^{-1}n^{1/3}, An^{1/3}]]< \epsilon\;.
$$
and 
$$
\Pro[T_{\mix}(\mathcal{C}) \notin [A^{-1}n, An]]< \epsilon\;.
$$
\end{corollary}

%

\medskip

\new{\textbf{Organisation of the paper.}} The paper is organized as follows. 
In Section~\ref{sec:explo}, we describe our exploration process of $G_{n,d,p}$ and introduce different quantities we will track during the process.
In Section~\ref{sec:swi}, we present our main combinatorial tool (switching method) and prove two technical lemmas. 
In Section~\ref{sec:analy}, we use these lemmas to study a single step of the exploration process. 
Finally, in Section~\ref{sec:proof}, we conclude with the proof of the lower bound in Theorem~\ref{thm:main}.

\section{The exploration process}\label{sec:explo}

Before describing the exploration process, we briefly introduce some notation.
For a graph~$G$, a subset of vertices $X$ of $G$, and a vertex $u$ of $G$,
we write $d_G(u)$ for the number of neighbours of $u$ in $G$ and $d_{G,X}(u)$ for the number of neighbours of $u$ in $G$ that belong to~$X$.
We also write $\Delta(G)$ for the maximum degree of $G$. Finally, for $p\in [0,1]$, we write $G_p$ for the graph where each edge in $G$ is independently retained with probability $p$.

We will use an exploration process to reveal the component structure of $G_{n,d,p}$.
Let us denote the vertex set by $V$, which we equip with a linear order (from now on $V$ is always a vertex set of size $n$).
For technical reasons, we perform our exploration process not on $G_{n,d,p}$, but on what we call an input.
An \emph{input} is a tuple $(G,\fS)$,
where $G\in \mathcal{G}_{n,d}$ and $\fS=\{\sigma_v\}_{v\in V}$ is a collection of $n$ permutations of length $d$.
For each vertex of $G$, arbitrarily label the edges incident to it with distinct elements from $\{1,\ldots, d\}$. 
Thus every edge receives two labels. In fact, we may think about this as a labelling of the semi-edges of $G$.
Let $\cI$ be the set of all inputs $(G,\fS)$ where $G\in \cG_{n,d}$ and $\fS$ is a collection of $n$ permutations of length $d$. 
Observe that every graph in $G\in \cG_{n,d}$ gives rise to exactly $(d!)^n$ inputs.
Thus, choosing an input uniformly at random from $\cI$ and ignoring the edge-labels is equivalent to choosing $G_{n,d}$.
Let $\fS_{n,d}$ be a collection of $n$ permutations of length~$d$ each chosen independently and uniformly at random. 
Hence, if an input is chosen uniformly at random from $\cI$, then this input is distributed as $(G_{n,d},\fS_{n,d})$.
\smallskip

Next, we describe our exploration process on an input $(G,\fS)$.
First, for every $uv\in E(G)$, we denote by $I(uv)$ the indicator random variable that is $1$ if $uv$ belongs to $G_p$ (it percolates) and $0$ otherwise. If $I(uv)$ is revealed, we say that the edge $uv$ has been exposed.
For each integer $t\geq 0$, the set $S_t$ consists of the vertices explored up to time $t$ (with $S_0=\emptyset$); 
the bipartite graph $F_t$, with bipartition $(S_t,V\sm S_t)$, consists of all edges in $G$ between $S_t$ and $V \sm S_t$ that have been exposed and have failed to percolate; and the graph $H_t$, with vertex set $S_t$, consists of all edges in $G$ within $S_t$, that is, $H_{t}:=G[S_{t}]$. 
Let $\cH_t$ be the history of all random choices we make until time $t$ (which we will treat as an event).

We now describe how to obtain $\cH_{t+1}$, given $\cH_{t}$.
Suppose there exists at least one vertex $u\in S_t$ such that $d_{H_t}(u)+d_{F_t}(u)<d$.
Among all such vertices $u$, let $v_{t+1}$ be the vertex which comes first in the linear order of $V$.
Let $w_{t+1}$ be the vertex $w\in V \sm {S_t}$ with $v_{t+1}w\in E(G)\sm E(F_t)$ that
minimizes $\sigma_{v_{t+1}}(\ell(w))$, where $\ell(w)$ is the label of the semi-edge incident to $v_{t+1}$ that 
corresponds to $v_{t+1}w$.
Thereafter, we 
expose $v_{t+1}w_{t+1}$.
If $I(v_{t+1}w_{t+1})=0$,
then we set $S_{t+1}:=S_t$, \new{$Y_{t+1}:=0$, $Z_{t+1}:=0$} 
and we let $F_{t+1}$ be the graph obtained from $F_t$ by adding $v_{t+1}w_{t+1}$.
If $I(v_{t+1}w_{t+1})=1$,
then we set
\begin{align*}
	S_{t+1}:=S_t\cup \{w_{t+1}\},\enspace 
	Y_{t+1}:=d_{F_t}(w_{t+1}) ,\enspace
	Z_{t+1}:=d_{G,S_t}(w_{t+1})-Y_{t+1}-1,\enspace 
\end{align*}
and we let $F_{t+1}$ be the graph obtained from $F_t$ by deleting all edges incident to $w_{t+1}$ and moving $w_{t+1}$ to the other side of the bipartition.
\new{Since $H_{t+1}=G[S_{t+1}]$, we also reveal all the edges between $w_{t+1}$ and $S_t$.}
Observe that $Z_{t+1}$ counts the number of neighbours of $w_{t+1}$ in $S_t\sm \{v_{t+1}\}$ whose corresponding edge has not yet been exposed.

If  $d_{H_t}(u)+d_{F_t}(u)=d$ for all $u\in S_t$, that is, every edge incident to a vertex in $S_t$ has been exposed,
then we pick a vertex $x\in V\sm S_t$ that minimises $d_{F_t}(x)$
and set $w_{t+1}:=x$, $S_{t+1}:=S_t\cup \{w_{t+1}\}$, \new{$Y_{t+1}:=d_{F_t}(w_{t+1})$, $Z_{t+1}:=0$} and we let $F_{t+1}$ be the graph obtained from $F_t$ by deleting all edges incident to $w_{t+1}$ and by moving $w_{t+1}$ to the other side of the bipartition.
Observe that, in any of the above-mentioned cases, $|E(F_{t+1})|\leq |E(F_{t})|+1$ and hence $|E(F_{t})|\leq t$.
\smallskip

A crucial parameter of our exploration process
is the number of edges between $S_t$ and $V\sm S_t$ which have not yet been exposed:
\begin{align*}
	X_t:=\sum_{u\in S_t}(d -d_{H_t}(u)-d_{F_t}(u))\;.
\end{align*}
For the sake of simplicity, we define $\eta_{t+1}:=X_{t+1}-X_t$. If $X_t>0$, then 
\begin{align}\label{eq:change}
\eta_{t+1}&= -(1-I(v_{t+1}w_{t+1}))+I(v_{t+1}w_{t+1})(d-2- Y_{t+1} - 2Z_{t+1}) \;,
\end{align}
and if $X_t=0$, then
\begin{align}\label{eq:change2}
\eta_{t+1}&= d-Y_{t+1}\;.
\end{align}
\new{Note that $Y_{t+1}$ and $Z_{t+1}$ are measurable random variables given $\cH_t$ and thus $\eta_{t+1}$ is a predictable sequence with respect to $\cH_t$.}

\section{The switching method and some applications}\label{sec:swi}

In this section we explain the switching method and we present two simple applications.
In Lemma~\ref{lem:UB_prob} we use the switching method to bound the probability from above that two vertices are adjacent.
In Lemma~\ref{lem: back edges} we provide an upper bound on the expectation of the number of neighbours of a vertex in a specified set of vertices.
\smallskip

Let $G$ be a graph and let $x_1,x_2,x_3,x_4$ be distinct vertices of $G$.
Suppose $x_1x_2,x_3x_4\in E(G)$ and $x_1x_4,x_2x_3\notin E(G)$.
A \emph{switching} on the $4$-cycle $x_1x_2x_3x_4$ transforms $G$ into a graph $G'$ by deleting $x_1x_2,x_3x_4$ and adding $x_1x_4,x_2x_3$.
Observe that the degree sequence of $G$ is preserved by the switching. In particular, if $G$ is $d$-regular, then so is $G'$. 
Moreover, the switching operation is reversible: if $G$ can be transformed into $G'$ by a switching, then $G$ can be also obtained from $G'$ by a switching on the same $4$-cycle.
Finally, there is a natural way to extend the notion of a switching from graphs to inputs by simply preserving the labels on each semi-edge.


Switchings can be used to obtain bounds on the probability that $G_{n,d}$ satisfies a certain property.
Suppose $\cA,\cB$ are disjoint subsets of $\cG_{n,d}$.
Suppose that for every graph $G\in \cA$, there are at least $a$ switchings that transform $G$ into a graph in $\cB$
and for every graph $G'\in \cB$, there are at most $b$ switchings that transform $G'$ into a graph in $\cA$.
By double-counting the number of switchings between $\cA$ and $\cB$, we obtain 
$a|\cA| \leq b |\cB|$.
Thus $a\Pro[\cA]\leq b\Pro[\cB]$,
where we define $\Pro[\cS]:=|\cS|/{|\mathcal{G}_{n,d}|}$ for every $\cS\subseteq \cG_{n,d}$.

\begin{lemma}\label{lem:UB_prob}
Suppose $d,n\in \N$ such that $3\leq d \leq n/4$ and  $S\subseteq V$ such that $|S|\leq n/6$. 
Let $H$ be a graph with vertex set $S$ and let 
$F$ be a bipartite graph with vertex partition $(S,V\sm S)$ with $\Delta(F\cup H)\leq d$.
Let $u\in S$ and $v\in V\sm S$ such that  $uv\notin E(F)$.
Then 
\begin{align*}
	\Pro[uv\in E(G_{n,d})\mid G_{n,d}[S]=H,\, F\subseteq G_{n,d}]\leq \frac{6(d-d_H(u)-d_F(u))}{n}\;.
\end{align*}
\end{lemma}
\begin{proof}
Let $\cF^+$ be the set of graphs $G\in \cG_{n,d}$ such that $G[S]=H$, $F\subseteq G$ and $uv\in E(G)$,
and let $\cF^-$ be the set of graphs $G\in \cG_{n,d}$ such that $G[S]=H$, $F\subseteq G$ but $uv\notin E(G)$. 
We will only perform switchings that involve edges and non-edges that are not contained in $E(H)\cup E(F)$. 
This ensures that the graph $G'$ obtained from a switching also satisfies $G'[S]=H$ and $F\subseteq G'$. 

Suppose $G\in \cF^+$. 
In order to bound the number of switchings from below it suffices to switch on a cycle $uvxy$ that satisfies $xy\in E(G)$, $uy,vx\notin E(G)$, and $x,y\in V\sm S$.
There are at least $dn-2d|S|$ ordered edges $xy$ with both endpoints in $V\sm S$. 
There are at most $d^2$ edges $xy$ such that $x$ is at distance at most $1$ from $v$ and at most $d^2$ edges $xy$ such that $y$ is at distance at most $1$ from $u$. Thus, there are at least $dn-2d|S|-2d^2 \geq dn/6$ switchings that transform $G$ into a graph in $\cF^-$.
Suppose now $G\in \cF^-$. 
Then there are clearly at most $d\cdot(d-d_H(u)-d_F(u))$ switchings that transform $G$ into a graph in $\cF^+$. 
It follows that
\begin{align*}
	\Pro[uv\in E(G_{n,d}) &\mid G_{n,d}[S]=H,\, F\subseteq G_{n,d}]\\
	&\leq \frac{d(d-d_H(u)-d_F(u))}{dn/6}\cdot \Pro[uv\notin E(G_{n,d})\mid G_{n,d}[S]=H,\, F\subseteq G_{n,d}]\\
	&\leq \frac{6(d-d_H(u)-d_F(u))}{n}\;.
	\qedhere
\end{align*}
\end{proof}

\begin{lemma}\label{lem: back edges}
Suppose $d,n\in \N$ such that $3\leq d \leq n/4$ and $S\subseteq V$ such that $|S|\leq n/6$. 
Let $H$ be a graph with vertex set $S$ and let
$F$ be a bipartite graph with vertex partition $(S,V\sm S)$ with $\Delta(F\cup H)\leq d$.  
Let $v\in V\sm S$.
Then
\begin{align*}
	\Exp[d_{G,S}(v)-d_{F}(v) \mid G_{n,d}[S]=H,\; F\subseteq G_{n,d}]\leq 6d|S|/n.
\end{align*}
\end{lemma}
\begin{proof}
For every $k\geq 0$, let $\cF_k$ be the set of graphs $G\in \cG_{n,d}$  such that $G[S]=H$, $F\subseteq G$, and
$d_{G,S}(v)-d_{F}(v)=k$. 
As in Lemma~\ref{lem:UB_prob}, we will only perform switchings using edges and non-edges that are not contained in  $E(H)\cup E(F)$.

Consider a graph in $\cF_k$. 
There are at most $(d-d_{F}(v))\cdot d|S|\leq d^2|S|$ switchings that lead to a graph in $\cF_{k+1}$.
For every graph in $\cF_{k+1}$,
we can use a switching on a cycle $uvxy$ that satisfies $uv,xy\in E(G)\sm E(F)$, $uy,vx\notin E(G)$ and $u\in S$,\ and  $v,x,y\in V\sm S$.
There are $k+1$ choices for $uv$ and, for any particular choice of $uv$,
there are at least $dn-2d|S|-2d^2\geq dn/6$ choices for the (ordered) edge $xy$.
Hence, there are at least $(k+1)dn/6$ switchings that lead to a graph in $\cF_k$.
Thus, for every $k\geq 0$, we obtain
\begin{align}\label{eq:bound}
	\Pro[\cF_{k+1}]
	\leq \frac{6d|S|/n}{(k+1)}\cdot \Pro[\cF_{k}]\;.
\end{align}
Let $X$ be a Poisson distributed random variable with mean $6d|S|/n$.
Lemma 3.4 in~\cite{Mrandom2012} together with~\eqref{eq:bound} implies that for every $m\geq 0$
$$
\Pro[d_{G,S}(v)-d_{F}(v)\geq m \mid  G_{n,d}[S]=H,\; F\subseteq G_{n,d}] \leq \Pro[X\geq m]\;,
$$
which implies the statement of the lemma.
\end{proof}

\section{Analysis of the exploration process}\label{sec:analy}

In this section we show how to
control the expectation of $\eta_t$ and $\eta_t^2$.
We first use Lemmas~\ref{lem:UB_prob} and~\ref{lem: back edges} to bound the expectation of $Y_{t+1}$ and $Z_{t+1}$ from above.

\begin{lemma}\label{lem:exp2}
Suppose $d,n\in \N$ such that $3\leq d\leq n-1$ and $n$ is sufficiently large. Fix $p\in [0,1]$.
Consider the exploration process described above on $(G_{n,d},\fS_{n,d})$
with percolation probability $p$ and suppose $t\leq d n^{2/3}$. 
Conditional on $\cH_t$ satisfying $|S_t|\leq 5n^{2/3}$, 
we have
\begin{align*}
\bE[Y_{t+1}|\cH_t]\leq 20dn^{-1/3} \text{\enspace and \enspace } \bE[Z_{t+1}|\cH_t]\leq 180dn^{-1/3}\;.
\end{align*}
\end{lemma}
\begin{proof}
If $\cH_t$ satisfies $X_t=0$, then
$Y_{t+1}\leq t/(n-|S_t|)\leq 2dn^{-1/3}$ by our choice of $w_{t+1}$ 
(we always choose the vertex $x$ that minimises $d_{F_t}(x)$)
and $|E(F_t)|\leq t$.
Note that $Z_{t+1}=0$ by definition.
Hence we may assume from now on that $X_t>0$.

Note that if $d\geq n/4$, then the lemma follows directly from the fact that  $Y_{t+1}\leq |S_t|\leq 5n^{2/3} \leq 20dn^{-1/3}$, and similarly for $Z_{t+1}$. Thus, in the following we assume that $d\leq n/4$.

\smallskip

Given $w\in V\sm S_t$ such that $v_{t+1}w\notin E(F_t)$, we apply Lemma~\ref{lem:UB_prob} with $S=S_t$, $F=F_t$, $H=H_t$, $u=v_{t+1}$ and $v=w$ to obtain
$$
\Pro[v_{t+1}w\in E(G_{n,d}) \mid v_{t+1}w\notin E(F_t),\cH_t]\leq \frac{6(d-d_{H_t}(v_{t+1})-d_{F_t}(v_{t+1}))}{n}\;.
$$
Observe that we run our exploration process on inputs.
In order to apply Lemma~\ref{lem:UB_prob}, we fix the semi-edge labelings and perform switchings on the graphs.

Since $\sigma_{v_{t+1}}$ is a random permutation, each edge incident to $v_{t+1}$ that is not contained in $E(F_t)\cup E(H_t)$ is chosen with the same probability to continue the exploration process. Hence, given that $v_{t+1}w\in E(G_{n,d})\sm E(F_t)$, the probability that $w_{t+1}=w$ is precisely $(d-d_{H_t}(v_{t+1})-d_{F_t}(v_{t+1}))^{-1}$.
Therefore,
\begin{align*}
	&\quad\,\, \Pro[w_{t+1}=w\mid v_{t+1}w\notin E(F_t),\cH_t]\nonumber \\
	&=\Pro[w_{t+1}=w\mid v_{t+1}w\in E(G_{n,d})\sm E(F_t),\cH_t]\cdot\Pro[v_{t+1}w\in E(G_{n,d}) \mid v_{t+1}w\notin E(F_t),\cH_t]\nonumber  
	\leq  \frac{6}{n}\;.
\end{align*}
Since $\Pro[w_{t+1}=w\mid v_{t+1}w\in E(F_t) , \cH_t]=0$, it follows that for every $w\in V\sm S_t$
\begin{align}\label{eq:prob}
\Pro[w_{t+1}=w\mid \cH_t]&\leq \frac{6}{n}\;.
\end{align}
Using that $|E(F_t)|\leq t$, we conclude
\begin{align*}
\bE[Y_{t+1}|\cH_t] 
&= \sum_{w\in V\sm S_t} d_{F_t}(w)\Pro[w_{t+1}=w|\cH_t] \stackrel{\eqref{eq:prob}}{\leq} \frac{6}{n}\sum_{w\in V\sm S_t} d_{F_t}(w)
 \leq \frac{6}{n}\cdot t\leq 6d n^{-1/3}\;.
\end{align*}
We now prove the second statement. Given $w\in V\sm S_t$ with  $\Pro[w_{t+1}=w\mid \cH_t]>0$ (that is, $v_{t+1}w\notin E(F_t)$), 
we apply Lemma~\ref{lem: back edges} with $S=S_t$, $F$ obtained from $F_t$ by adding $v_{t+1}w$, $H=H_t$, and $v=w$, to obtain 
\begin{align*}
\bE[Z_{t+1}|\cH_t]
&= \sum_{w\in V\sm S_t} \bE[Z_{t+1}|w_{t+1}=w,v_{t+1}w\notin E(F_t),\cH_t]\Pro[w_{t+1}=w\mid v_{t+1}w\notin E(F_t), \cH_t]  \\
&\stackrel{\eqref{eq:prob}}{\leq} \sum_{w\in V\sm S_t} \frac{6d|S_t|}{n} \cdot \frac{6}{n}
\leq 180d n^{-1/3}\;.\qedhere
\end{align*}
\end{proof}

\begin{lemma}\label{lem:exp}
Suppose $\mu\geq 0$ and  $d,n\in \N$  such that $3 \leq d\leq n-1$ and $n$ is sufficiently large. 
Consider the exploration process described above on $(G_{n,d},\fS_{n,d})$ with $p=\frac{1-\mu n^{-1/3}}{d-1}$ and suppose $t\leq d n^{2/3}$. 
Conditional on $|S_t|\leq 5n^{2/3}$, then 
\begin{align*}
	\bE[\eta_{t+1}|\cH_t]\geq -(570+\mu)n^{-1/3} \enspace \text{ and }\enspace
\bE[\eta_{t+1}^2|\cH_t] \geq d/4\;.
\end{align*}
Moreover, if $X_t>0$, then $\bE[\eta_{t+1}^2|\cH_t] \leq d$.
\end{lemma}
\begin{proof}
First assume that $X_t>0$. Recall that for any $\cH_t$ \new{and for any edge $uv$ that has not been exposed yet}, we have $\bE[I(uv)\mid \cH_t]=\new{p=}(1-\mu n^{-1/3})/(d-1)$. 
\new{Recall that $Y_{t+1}$ and $Z_{t+1}$ are measurable with respect to $\cH_t$.} Taking conditional expectations on~\eqref{eq:change} and using Lemma~\ref{lem:exp2}, we obtain
\begin{align*}
\bE[\eta_{t+1}|\cH_t]
&=-\left(1-\frac{1-\mu n^{-1/3}}{d-1}\right) + \frac{1-\mu n^{-1/3}}{d-1}(d-2 -\bE[Y_{t+1}|\cH_t]-2\bE[Z_{t+1}|\cH_t])\\
&\geq -\frac{\bE[Y_{t+1}|\cH_t]+2\bE[Z_{t+1}|\cH_t]}{d-1} -\mu n^{-1/3}\\
&\geq -\frac{380 dn^{-1/3}}{d-1} -\mu n^{-1/3}
\geq -(570+\mu)n^{-1/3}\;,
\end{align*}
since $d\geq 3$.

Again, by Lemma~\ref{lem:exp2} and~\eqref{eq:change}, we obtain
\begin{align*}
\bE[\eta_{t+1}^2|\cH_t]
&=\left(1-\frac{1-\mu n^{-1/3}}{d-1}\right)(-1)^2+ \frac{1-\mu n^{-1/3}}{d-1}\bE[(d-2 -Y_{t+1}-2Z_{t+1})^2\mid\cH_t]\\
&\geq \frac{d-2}{d-1}+  \frac{(1-\mu n^{-1/3})(d-2)^2}{d-1} -\frac{2(d-2)(\bE[Y_{t+1}|\cH_t]+2\bE[Z_{t+1}|\cH_t])}{d-1}\\
&\geq (1-\mu n^{-1/3})(d-2) -2(\bE[Y_{t+1}|\cH_t]+2\bE[Z_{t+1}|\cH_t])\\
&\geq (1-\mu n^{-1/3})(d-2) -760dn^{-1/3}\\
&\geq d/4\;,
\end{align*}
where the last inequality holds since $d\geq 3$ and $n$ is sufficiently large.
Observe that $\bE[\eta_{t+1}^2|\cH_t]\leq d$ follows from a similar argument as $(d-2 -Y_{t+1}-2Z_{t+1})^2\leq (d-2)\new{^2}$.

If $X_t=0$, then clearly $\bE[\eta_{t+1}|\cH_t]\geq 0$ and, 
since $\bE[\eta_{t+1}^2|\cH_t]=\bE[(d-Y_{t+1})^2|\cH_t]$, 
similarly as before, one can prove that $\bE[\eta_{t+1}^2|\cH_t]\geq d/4$. 
\end{proof}

\begin{lemma}\label{lem:S_t}
Suppose $\mu\geq 0$ and $d,n\in \N$ such that $3\leq d\leq n-1$ and $n$ is sufficiently large. 
Consider the exploration process described above on $(G_{n,d},\fS_{n,d})$ with $p=\frac{1-\mu n^{-1/3}}{d-1}$. 
Then, for every fixed $\delta>0$ and all $0\leq t_1\leq t_2 \leq 5dn^{2/3}$, we have
\begin{align*}
&\Pro\left[ |S_{t_2}\sm S_{t_1}| - \frac{t_2-t_1}{d-1} \geq -\delta n^{2/3} \right]= \new{1-o(n^{-2})}\quad\text{ and}\\
&\Pro\left[ |S_{t_2}\sm S_{t_1}| - \frac{t_2-t_1}{d-1} - \left\lceil\frac{\new{t_2}}{5d/6}\right\rceil\leq \delta n^{2/3} \right]= \new{1-o(n^{-2})}\;.
\end{align*}
\end{lemma}

\begin{proof}
We add a vertex to $S_t$ either if $I(v_{t+1}w_{t+1})=1$ or if we start exploring a new component of $G_{n,d,p}$ at time $t+1$.
Thus, \new{$|S_{t_2}\sm S_{t_1}|$} stochastically dominates a binomial random variable with parameters \new{$t_2-t_1$} and $(1-\mu n^{-1/3})/(d-1)$.
A standard application of Chernoff's inequality implies the first statement.

Let $A_t\subseteq S_t$ be the set of vertices that start a new component in $G_{n,d,p}$. 
\new{For every $0\leq t\leq 5dn^{2/3}$, let $a_t:=|A_t|$, let $c_t:=|S_t\sm A_t|$ and let $b_t:=|S_t\sm (S_{t_1}\cup A_t)|$}. 
Observe that \new{$c_t$} is stochastically dominated by a binomial random variable with parameters $t$ and $1/(d-1)$. 
Using Chernoff's inequality, \new{we have $ c_t\leq 8 n^{2/3}$} with probability \new{$1-o(n^{-2})$} for any $t\leq 5dn^{2/3}$.

We claim that for every \new{$0\leq t\leq 5dn^{2/3}$ and conditional on $c_t\leq 8n^{2/3}$}, we have $a_t\leq \lceil\frac{t}{5d/6}\rceil$. 
Indeed, the claim is true for $t\in \{0,1\}$.
Assume that $t\geq 2$ and that the claim holds for every $t'\in \{0,\ldots,t-1\}$. 
If $X_{t-1}>0$, then $a_t=a_{t-1}$ and we are done. Thus, assume that $X_{t-1}=0$.
Let $s$ be the largest integer $s'\in \{0,\ldots,t-2\}$ such that $X_{s'}=0$ (it exists since $X_0=0$ and $t\geq 2$).
Recall that $w_{s+1}$ is a vertex $x\in V\sm S_s$ that minimises $d_{F_s}(x)$. It follows that
$$
d_{F_{s}}(w_{s+1})\leq \frac{|E(F_{s})|}{n-(a_s+\new{c_s})}\leq \frac{s}{n- \lceil s/(5d/6)\rceil- 8n^{2/3}}\leq \frac{d}{6}\;,
$$ 
provided that $n$ is large enough. 
Hence, $X_{s+1}\geq 5d/6$ and the process will not start a new component for the next $5d/6$ steps.
In particular, $s+  5d/6\leq t$.
This implies $a_{t}= a_s+1\leq \lceil\frac{s}{5d/6}\rceil+1\leq \lceil\frac{t}{5d/6}\rceil$.

\new{ 
Since $|S_{t_2}\sm S_{t_1}|\leq a_{t_2}+b_{t_2}$, 
the second part of the lemma now follows from the upper bound on $a_{t_2}$ (which holds as we assume $c_t\leq 8n^{2/3}$) and 
an upper bound on $b_{t_2}$ obtained by Chernoff's inequality.} 
\end{proof}

\section{Proof of Theorem~\ref{thm:main}}\label{sec:proof}

As we mentioned in the introduction, due to the result of Nachmias and Peres, we only need to prove a lower bound. 
Since it suffices to prove the lower bound of the statement for $\lambda\leq 0$, we use the definition $\mu:=-\lambda$.
We now present a brief overview of the proof. 
In the first phase, we show that with probability at least $1-A^{-1/2}$, 
the process $X_t$ exceeds $A^{-1/4} dn^{1/3}$ in the first $dn^{2/3}$ steps.
In the second phase and conditional on the success of the first phase, we show that  
$X_t$ stays positive for at least $2A^{-1}dn^{2/3}$ steps with probability at least $1-A^{-1/2}$. 
From standard concentration inequalities, this gives the existence of a component of order at least $A^{-1}n^{2/3}$, 
concluding the proof.
This proof strategy was introduced by Nachmias and Peres to prove the same statement for fixed $d\geq 3$~\cite{nachmias2010critical} and for $d=n-1$~\cite{nachmias2010critical2}.
We remark that, in comparison to \cite{nachmias2010critical}, 
our analysis of the exploration process is simpler, as we do not need to track the number of vertices $x\in V\sm S_t$ which satisfy $d_{F_t}(x)=k$
for $k\in \{0,1,\ldots,d\}$.
If $d\geq 3$ is fixed, as in~\cite{nachmias2010critical}, almost every vertex $x$ satisfies $d_{F_t}(x)\in \{0,1\}$.
However, this is no longer true if $d$ is an arbitrary function of $n$. 
We avoid the technicalities involved with this issue by averaging over the values of $d_{F_t}(x)$.

\medskip

\noindent\textbf{First phase:}
We start with the definition of a few parameters.
Let $h:=A^{-1/4} dn^{1/3}$, $T_1:= 5dn^{2/3}/6$ and $T_2:=2A^{-1} d n^{2/3}$. 
In addition, we define the following stopping times:
\begin{align*}
\tau_h&:=\min\{t:\, X_t\geq h\}\wedge T_1\\
\tau_S^1&:=\min\{t:\, |S_t|\geq  3n^{2/3}\}\\
\tau_1&:= \tau_h\wedge \tau_S^1\;.
\end{align*}
Recall that $X_{t+1}=\eta_{t+1}+X_t$. 
Note also that for every $t< \tau_1$, we have $X_t\leq h$ and $|S_t|\leq 5n^{2/3}$. 
Hence, Lemma~\ref{lem:exp} implies that
\begin{align*}
\bE[X_{t+1}^2-X_t^2|\cH_t] 
&\geq \bE[\eta_{t+1}^2|\cH_t] +2\bE[\eta_{t+1} X_{t}|\cH_t] 
\geq d/4- 2\cdot(570+\mu) n^{-1/3} h\geq {d/5}\;,
\end{align*}
provided that $A$ is large enough with respect to $\mu$ (and thus, with respect to $\lambda$).
Hence $X_{t\wedge \tau_1}^2 - (t\wedge \tau_1)d/5$ is a submartingale. 
By the Optional Stopping theorem for submartingales (see for example~\cite{GrimStirz} p.491), 
$\bE[X_{\tau_1}^2 - \frac{d}{5}\tau_1]\geq \bE[X_{0}^2] = 0$, 
which implies that $\bE[\tau_1]\leq \frac{5}{d}\bE[X_{\tau_1}^2]$.
Since $X^2_{\tau_1}\leq (h+d)^2\leq 2 h^2$,	we obtain
$$
\Pro[\tau_1=T_1]\leq \frac{\bE[\tau_1]}{T_1} \leq \frac{5\bE[X_{\tau_1}^2]}{d T_1}\leq\frac{10 h^2}{dT_1} =  12A^{-1/2}\;.
$$
By Lemma~\ref{lem:S_t} \new{with $t_1=0$ and $t_2=T_1$}, we have $\Pro[\tau_S^1\leq T_1]=o(1)$.
Thus
\begin{align}\label{eq:T1}
\Pro[\{\tau_h=T_1\} \cup \{\tau_S^1\leq \tau_h\}]
\leq \Pro[\tau_1=T_1]+\Pro[\tau_S^1\leq T_1]
\leq  12A^{-1/2}+o(1) \leq  13A^{-1/2}\;.
\end{align}
We conclude that \new{the event} $\new{\mathcal{E}:=}\{\tau_h<T_1,\tau_h<  \tau^1_S\}$ holds with probability at least $1-13A^{-1/2}$. 
In particular, with probability at least $1-13A^{-1/2}$, the random process $X_t$ exceeds $h$ before time $T_1$.
\smallskip

\noindent\textbf{Second phase:} 
Write $\Pro_*$ and $\bE_*$ for the probability and the expectation conditional on $\new{\mathcal{E}}$. We define
\begin{align*}
\tau_0:&=\min\{t: X_{\tau_h+t}=0\}\wedge T_2\\
\tau^2_S:&=\min\{t: |S_{\tau_h+t}\sm S_{\tau_h}|\geq \new{2}n^{2/3}\}\\
\tau_2:&=\tau_0\wedge \tau_S^2\;.
\end{align*}
Consider the random variable 
$$
W_t:=h- \min \{h, X_{\tau_h+t}\}
\;.
$$
Hence
\begin{align*}
	W_{t+1}^2-W_{t}^2
	&\leq (h- \min\{h,X_{\tau_h+t}\}-\eta_{\tau_h+t+1})^2-(h- \min\{h,X_{\tau_h+t}\})^2\\
	&= \eta_{\tau_h+t+1}^2-2 \eta_{\tau_h+t+1}(h- \min\{h,X_{\tau_h+t}\})\\
	&\leq \eta_{\tau_h+t+1}^2 -2 \eta_{\tau_h+t+1}h	\;.
\end{align*}
If $t<\tau_2$ and $n$ is sufficiently large, we can apply Lemma~\ref{lem:exp} and this leads to (provided $A$ is sufficiently large with respect to $\mu$)
\begin{align*}
	\Exp_*[W_{t+1}^2-W_{t}^2\mid \cH_{\tau_h+t}]
	\leq d+2 \cdot (570+\mu)n^{-1/3} \cdot h
	\leq 2d\;.
\end{align*}
Thus, $W_{t\wedge \tau_2}^2-2d(t\wedge\tau_2)$ is a supermartingale.
Similar as before, we use the Optimal Stopping theorem to conclude that
\begin{align*}
	\Exp_*[W_{\tau_2}^2]
	\leq 2d\Exp_*[\tau_2]
	\leq 2dT_2 \;.
\end{align*}
Thus
\begin{align*}
	\Pro_*[\tau_2<T_2]\notag
	&= \Pro_*[\tau_0<T_2,\tau^2_S>T_2]+\Pro_*[\tau^2_S\leq T_2]\\
	&\leq \Pro_*[W_{\tau_2}\geq h]+\Pro_*[|S_{\tau_h +T_2}\sm S_{\tau_h}|\geq \new{2}n^{2/3}]\\
	&\leq \Pro_*[W^2_{\tau_2}\geq h^2]+o(1)\\
	&\leq \frac{\Exp_*[W_{\tau_2}^2]}{h^2}+o(1)\leq 5A^{-1/2}\;,
\end{align*}
where we used Lemma~\ref{lem:S_t} \new{with $t_1=\tau_h$ and $t_2=\tau_h+T_2$} for the second inequality.
\new{(Observe that we cannot apply Lemma~\ref{lem:S_t} directly, because
we assume $\mathcal{E}$ holds and $\tau_h$ is a random time. 
However, as $\tau_h\leq  T_1$, a simple union bound with $t_1=k$ and $t_2=k+T_2$ for all $k\leq T_1$ 
together with the fact that $\Pro[\mathcal{E}]\geq 1-13A^{-1/2}\geq 1/2$, 
yields the desired result.)}
It follows that
\begin{align*}
	\Pro[\{\tau_2<T_2\}\cup \{\tau_h=T_1\} \cup \{\tau_S^1\leq \tau_h\}]
	&\leq \Pro[\{\tau_h=T_1\} \cup \{\tau_S^1\leq \tau_h\}]+ \Pro_*[\tau_2<T_2]\\
	&\stackrel{\eqref{eq:T1}}{\leq}  13A^{-1/2} +5A^{-1/2}=18A^{-1/2}\;.
\end{align*}

Since all the vertices explored from time $\tau_h$ to $\tau_h+\tau_2$ belong to the same component of $G_{n,d,p}$, 
there exists a component of size at least $|S_{\tau_h+\tau_2}\sm S_{\tau_h}|$. 
As $\tau_2=T_2= 2A^{-1} dn^{2/3}$ with probability at least $1-18A^{-1/2}$, 
by Lemma~\ref{lem:S_t} \new{with $t_1=\tau_h$ and $t_2=\tau_h+T_2$ 
(as above, strictly speaking, we apply Lemma~\ref{lem:S_t}
with $t_1=k$ and $t_2=k+T_2$ for all $k\leq T_1$ and use the fact that $\Pro[\mathcal{E}]\geq 1/2$}) 
with probability at least $1-18A^{-1/2}-o(1)\geq 1-19A^{-1/2}$, 
there exists a component of size at least $A^{-1} n^{2/3}$.

\medskip

\textbf{Acknowledgements:} The authors want to thank Nikolaos Fountoulakis, Michael Krivelevich, and Asaf Nachmias for fruitful discussions on the topic
\new{and the anonymous referees for their valuable comments.}

\bibliographystyle{amsplain}
\bibliography{crit_ref}

\vfill

\small
\vskip2mm plus 1fill
\noindent
Version \today{}
\bigbreak

\noindent
Felix Joos\\
{\tt <f.joos@bham.ac.uk>}\\
Guillem Perarnau\\
{\tt <g.perarnau@bham.ac.uk>}\\
School of Mathematics, University of Birmingham, Birmingham\\
United Kingdom

\end{document}